\documentclass[12pt,leqno]{amsart}
\usepackage{amsfonts,amsmath,amssymb,amsthm, geometry}

\usepackage[all]{xy}

\newtheorem{theorem}{Theorem}[section]
\newtheorem{corollary}[theorem]{Corollary}
\newtheorem{lemma}[theorem]{Lemma}
\newtheorem{proposition}[theorem]{Proposition}
\theoremstyle{definition}
\newtheorem{definition}[theorem]{Definition}
\theoremstyle{remark}
\newtheorem{remark}[theorem]{\sc Remark}
\newtheorem{example}[theorem]{\sc Example}
\renewcommand{\Box}{\square}    %\diamond

\newcommand{\diffeo}{{\rm{diffeo}}}

\newcommand{\proj}{{\rm{proj}}}

\newcommand{\Sing}{{\rm{Sing\hspace{2pt}}}}
\newcommand{\rank}{{\rm{rank\hspace{2pt}}}}

\newcommand{\ity}{{\infty}}

\newcommand{\e}{\varepsilon}
\newcommand{\m}{\setminus}
\newcommand{\fin}{\hspace*{\fill}$\Box$\vspace*{2mm}}

\newcommand{\cS}{{\mathcal S}}

\newcommand{\bR}{{\mathbb R}}
\newcommand{\bC}{{\mathbb C}}
\newcommand{\bP}{{\mathbb P}}
\newcommand{\bN}{{\mathbb N}}

\newcommand{\bv}{{\bf{v}}}
\newcommand{\bw}{{\bf{w}}}

\begin{document}
\title[Real map germs and open books]{Real map germs and higher open books}
%Open book decompositions and fibrations of real map germs

\author{\sc Raimundo Ara\'ujo dos Santos}
\address{Instituto de Ci\^encias Matem\'aticas e de Computa\c c\~ao,
Universidade de S\~ao Paulo,  Av. Trabalhador S\~ao-Carlense, 400 -
Centro Postal Box 668, S\~ao Carlos - S\~ao Paulo - Brazil
\newline Postal Code 13560-970, S\~ao Carlos, SP, Brazil.}
\email{rnonato@icmc.usp.br}
\author{\sc Mihai Tib\u ar}
\address{Math\'ematiques, UMR-CNRS 8524, Universit\'e des Sciences et Technologies de Lille,
59655 Villeneuve d'Ascq, France.}
\email{tibar@math.univ-lille1.fr}



\date{December 30, 2007}
\maketitle
%%%%%%%%%%%%%%%%%%%
\begin{abstract}
We present a general criterion for the existence of open book structures
defined by real map germs $(\bR^m ,0) \to (\bR^p ,0)$, where $m> p \ge 2$,
with isolated critical point. We show that this is satisfied by weighted-homogeneous maps.
We also derive sufficient conditions in case of map germs with isolated critical value.
\end{abstract}

\section{Introduction and results}\label{intro}

An \textit{open book structure} of a real manifold $M$ is a pair $(K, \theta)$
 consisting of a codimension two submanifold $K$ and a locally trivial smooth fibration $\theta : M\setminus K \to S^1$ over the circle $S^1 :=\partial B^2$ such that $K$ admits neighbourhoods $N$ diffeomorphic to $B^2 \times K$ 
for which $K$ is identified to $\{ 0\} \times K$ and the restriction  
$\theta_{|N\setminus K}$ is the following composition with the natural projections $N\setminus K  \stackrel{\diffeo}{\simeq} (B^2 \setminus \{ 0\}) \times K \stackrel{\proj}{\to} B^2 \setminus \{ 0\} \to S^1$. One says that $K$ is the \textit{binding} and that the (closures of) the fibers of $\theta$ are the \textit{pages} of the \textit{open book}.
  
Open book structures are known since long time but the term appears rather recently
in the paper by Winkelnkemper \cite{Wi}. They appear in the literature under various names: fibered links, Milnor or Lefschetz fibrations, spinnable structures, etc. Most recently, they occur as fundamental objects in symplectic and contact geometry, beginning with the pioneering studies by Arnold, Donaldson, Giroux \cite{Gi} and other authors, see e.g. the survey \cite{Et}.

We focus here on the existence of fibered links defined by
real analytic map germs. The complex setting has been intensely investigated by Milnor and Brieskorn in the 1960's and by many other authors ever since; this gives rise in particular to rich open book structures.

In his Princeton studies \cite{Mi} of complex hypersurface singularities,
Milnor included a short section about real analytic map germs
$\psi :(\mathbb R^m ,0) \to (\mathbb R^p ,0)$, where $m\ge p \ge 2$, in order to show the elusiveness of the real setting compared to complex one.
Let $V := \{ \psi = 0 \}$ denote the germ at $0$ of the zero locus of $\psi$, let
 $B^m_\e \subset \bR^m$ be the open ball of radius $\e$ at the origin
and $S^{m-1}_\e := \partial \overline{B_\e}$. We start from the following two results:

\smallskip
\noindent
\textbf{Fact 1.} It is shown in \cite[page 97-99]{Mi} that if $\psi$ has an isolated critical point at $0\in \mathbb{R}^{m}$ then for any small enough $\e >0$, the complement $S^{m-1}_\e \m K_\e$ of the link $K_\e := S_\e \cap V$ is the total space of a smooth fiber bundle over the unit sphere $S^{p-1}$.
Whenever $K_\e$ is nonempty, one says that it is a \textit{fibered link} (see e.g. \cite{Du}).  Moreover, Milnor's proof shows that the sphere $S^{m-1}_\e$ is endowed with an \textit{open book decomposition with \textit{binding} $K_\e$} as in the above definition, extended in the sense that the binding $K_\e$ is of \textit{higher codimension}\footnote{this also explains the occurrence of the word ``higher'' in the title of the paper} $p\ge 2$ instead of 2.

Milnor's proof gives the main construction of a fiber bundle:
\begin{equation}\label{eq:milnormap1}
 S^{m-1}_\e \m K_\e \to S^{p-1}
\end{equation}
 by integrating a vector field\footnote{for more details, see Step 2 of the proof of Theorem \ref{t:main1}}, without providing an explicit fibration map. One can only say (see the first part of Milnor's proof) that the fibration map is $\frac{\psi}{\|\psi \|}$ in some small enough neighbourhood $N_\e$ of $K_\e$.

\noindent
\textbf{Fact 2.} In case $p=2$ and $\psi$ is the pair of functions ``the real and the imaginary part'' of a holomorphic function germ $f : (\bC^n,0) \to (\bC,0)$, then Milnor \cite[\S 4]{Mi} proved that the natural map:
\begin{equation}\label{eq:milnormap2}
 \frac{\psi}{\|\psi \|} : S_\e^{m-1}\m K_\e \to S^{1}
\end{equation}
is itself a $C^\infty$ locally trivial fibration\footnote{called ``strong Milnor fibration'' by some authors, see e.g. \cite{RSV}.}, where $m$ is here $2n$. This explicit map provides another type of open book decomposition of the odd dimensional spheres $S_\e^{2n-1}$.

\smallskip
These facts raise the following two natural basic questions:

\smallskip
\noindent
(A). \textit{Given a real map $\psi$, under what more general conditions each of the two above fibrations exist ? In particular, which maps $\psi$ with isolated singularity possess a locally trivial fibration} (\ref{eq:milnormap2}) ?

\noindent
(B). \textit{Whenever both types of open book decompositions exist for a given $\psi$, are the
fibrations} (\ref{eq:milnormap1}) \textit{and} (\ref{eq:milnormap2}) isomorphic \textit{?}

\smallskip

The first necessary condition for the existence of both types of fibrations is 
the inclusion $\Sing \psi \subset V$.
This inclusion is true for instance in the setting of Fact 2 whenever $\psi$ is defined by a holomorphic function $f$. It is however not true for any real map germ $\psi$.

Regarding the question (A), Milnor exhibited an example with $m=p=2$ and $K= \emptyset$, cf  \cite[p. 99]{Mi}, showing that, even if the first type fibration exists, the map (\ref{eq:milnormap2}) might not be a fibration.
Nevertheless, more recently several authors obtained, in case $p=2$ and $\Sing \psi \subset \{ 0\}$, sufficient conditions under which (\ref{eq:milnormap2}) is a fibration \cite{Ja1, Ja2, RS, RSV, Se1, dS} and provided examples showing that the class of real map germs $\psi$ with isolated singularity satisfying them is larger than the class of holomorphic functions $f$.

 As for question (B), Milnor showed in \cite[\S 4]{Mi} that in case of a holomorphic germ $f$ with isolated singularity the two types of open book decompositions coincide. It is up to now not clear what happens beyond this class of maps and in particular what meaningful conditions should be  imposed to a map $\psi$ such that the answer be ``yes''.

We address here both questions.
Let us first observe that in order to become a smooth fibre bundle, the map (\ref{eq:milnormap2}) has to be a submersion\footnote{precisely this condition fails
 in Milnor's \cite[example p. 99]{Mi}.}. This condition would be also sufficient if  $\frac{\psi}{\|\psi \|}$ were a proper map, by the well-known Ehresmann's theorem.
This is the case when $V= \{ 0\}$, hence $K_\e = \emptyset$.

%Study the strange case $K_\e = \emptyset$. What is Im f  ?

 Nevertheless, the properness fails whenever $\dim V >0$ since $S_\e^{m-1}\m K_\e$ is not compact.
Our first result shows that, in case $\Sing \psi \subset \{ 0\}$, the above necessary condition is also sufficient, despite the non-properness of $\frac{\psi}{\|\psi \|}$.  This completes the answer to Milnor's question (i.e. the second part of question (A)), recovering the above cited results (see also Remark \ref{r:thmain}) and moreover providing a conceptual explanation for the general case $p\ge 2$.

%Moreover, we extend here the setting to nonisolated singularities: we replace the assumption ``isolated singular point'' by ``isolated critical value''.

\begin{theorem} \label{t:main1}
Let $\psi :(\mathbb R^m ,0) \to (\mathbb R^p ,0)$, for $m > p \ge 2$,  be a
real analytic map germ with isolated critical point at $0\in
\mathbb{R}^{m}$.

The map \rm (\ref{eq:milnormap2}) \it defines an open book decomposition of $S^{m-1}_\e$ with binding $K_\e$ if and only if this map is a submersion for any $\e >0$ small enough.
\end{theorem}

Examples where Theorem  \ref{t:main1} holds are \ref{e:2}, \ref{e:3} and last part of \ref{e:1}. Let us point out that Milnor's example \cite[p. 99]{Mi} does not satisfy the criterion of Theorem \ref{t:main1}. Indeed, by direct computations one may see that the equivalent condition given in Remark \ref{r:cond} fails, therefore (\ref{eq:milnormap2}) is not a submersion for small spheres. However Milnor's example has $\dim V =0$ so it does not fit the hypothesis $m>p$ of the above Theorem either. Let us give here a similar example which satisfies this hypothesis.

\begin{example}\label{e:1}
$\psi : (\mathbb R^3 ,0) \to (\mathbb R^2 ,0)$,
\[ \psi = (P,Q) :=
(y(x^2+y^2+z^2)+x^2, x). \]
The function $(P,Q)$ has isolated singularity at origin and the link is $K=\{x=y =0\}\cap S^{2}_{\e}$, i.e. two antipodal points.

We determine the set $M$ defined at (\ref{eq:main2})
by computing the $2\times 2$ minors of the $2\times 3$ matrix having $P\nabla Q-Q\nabla P$
on the first line and $(x,y,z)$ on the second line, where $\nabla P$ and $\nabla Q$ denote the gradients.
We find that $\bar M \cap \{ z= 0\} =  \{(x^2+y^2)^2=yx^2 \}$, which is a non-degenerate curve through the origin. This shows that the condition $\bar M \not\supset \{ 0\}$ fails and therefore, by Remark \ref{r:cond}) and
 Theorem \ref{t:main1}, the map (\ref{eq:milnormap2}) does not define an open book decomposition.

However, let us also remark that if instead of $P = y(x^2+y^2+z^2) + x^2$ we take $P = y(x^2+y^2+z^2)$ in the above example, i.e. taking the quadratic part out, then the condition $\bar M \not\supset \{ 0\}$ is satisfied.
\end{example}

Furthermore, we show in Section \ref{weighted-hom} that the criterion defined in Theorem \ref{t:main1}
is satisfied whenever $\psi$ is weighted-homogeneous, and that the fibrations (\ref{eq:milnormap1}) and (\ref{eq:milnormap2}) are equivalent.

In the last section we address the same questions (A) and (B) in the more general setting of maps
$\psi : (\bR^m, 0) \to (\bR^p,0)$ with an \textit{isolated critical value}. Related questions have been recently studied in this setting, see \cite{PS, Mas}.
We give sufficient conditions for the existence of the open book decompositions (\ref{eq:milnormap1}) and (\ref{eq:milnormap2}).

%%%%%%%%%%%%%%%%%%%%%%%%%%%%%%%%%%%%%%%%%%%%%
%%%%%%%%%%%%%%%%%%%%%%%%%%%%%%%%%%%%%%%%%%%%%%%%%%

\section{Preparations}\label{prep}

The following construction is inspired by the method introduced in Seade's paper \cite{Se1} and further used in \cite{RS}, and by the study of nongeneric pencils
of hypersurfaces (or meromorphic functions with singularities in base locus) after \cite{ST, ST2, Ti1, Ti2,  Ti}.

Let $D_\delta \subset \bR^p$ be a closed ball of radius $\delta >0$ at the origin.

Let $\psi(x)  = (P_1, \ldots, P_p)(x)$. Let $[s] := [s_1: \cdots : s_p]$ some point of the real projective space $\bP^{p-1}_{\bR}$. For some $\e_0 >0$ we define the following analytic set:
%%%%%%%%%%%%%%%%%%%%%%%%%%
\begin{equation}\label{eq:defX}
 X := \lbrace (x, [s_1: \cdots : s_p]) \in B^m_{\e_0}\times \bP^{p-1}_{\bR} \mid \rank
\left[ \begin{array}{ccc}
P_1(x) & \cdots & P_p(x) \\
s_1 & \cdots & s_p
\end{array} \right] < 2
 \rbrace ,
\end{equation}
which of course depends on the radius ${\e_0}$. Let
$\pi : X \to \bP^{p-1}_{\bR}$ be the natural projection and
let $X_{[s]}$ denote the fibre of $\pi$ over $[s]\in \bP^{p-1}$.
Let us observe that $X$ contains  $V \times \bP^{p-1}$.
For any $[s]\in \bP^{p-1}$, the set $X_{[s]}$ contains $V$. Actually $X_{[s]}\m V$ is the  disjoint union of two subsets separated by $V$, namely $X_{[s]}^+ := (\frac{\psi}{\| \psi\|})^{-1}(\frac{s}{\|s\|})$ and $X_{[s]}^- := (\frac{\psi}{\| \psi\|})^{-1}(-\frac{s}{\|s\|})$, where $\frac{s}{\|s\|}$ is the unit vector of the line represented by $[s]\in \bP^{p-1}$.

The projection on the first factor $p :  X \to B_{\e_0}$ is actually a blow-up of $B_{\e_0}$ along $V$. This map restricts to an analytic isomorphism:
\begin{equation}\label{eq:blowup}
 p_| : X \m (V \times \bP^{p-1}) \to B^m_{\e_0} \m V.
\end{equation}

\subsection{The singular locus of $X$.} \label{ss:sing}
 Let us first recall that the singular locus $\Sing \psi$ of $\psi$ is the set-germ at $0$ of the points $x\in \bR^m$, where the gradient vectors $\nabla P_1(x), \ldots , \nabla P_p(x)$ are linearly dependent.
 In some fixed chart $\{ s_i=1\}$ of $\bP^{p-1}$, the analytic set $X$ is defined by $p-1$
independent equations $f_{ij} := s_j P_i - P_j =0$, for $j\not= i$. Then the singular locus of
$X$ is by definition the set of points $(x, [s])\in X$ where the  gradients with respect to $x$ and to $s_j$ (for $j\not= i$) of the $p-1$ functions $f_{ij}$ are linearly dependent.

From these observations, by direct computations we get the following inclusions which are not equalities in general:

\begin{lemma}\label{l:singp}
 \[ \Sing X \subset  (V \cap \Sing \psi) \times \bP^{p-1}
\]
and, for any fixed $[s]\in \bP^{p-1}$:

\[ \Sing X_{[s]} \subset X_{[s]}\cap \Sing \psi.
 \]  \fin
 \end{lemma}
 This approximation is however enough for drawing the following immediate but useful  consequences. In case $p=2$, a sharper description of the singular loci is given in \S \ref{ss:case_p=2}.

\begin{corollary} \label{c:submanif} \
\begin{enumerate}
 \item If $\psi$ has an isolated critical point at $0$ then $\Sing X \subset \{ 0\} \times \bP^{p-1}$.
In particular the sets
 $(V \m \{ 0\}) \times \bP^{p-1}$
and $X_{[s]} \m (0,[s])$ are analytic submanifolds of $X \m (\{ 0\} \times \bP^{p-1})$.
\item If $\psi$ has an isolated critical value at $0$ then  $\Sing X \subset V \times \bP^{p-1}$.
In particular the sets
  $(V\times \bP^{p-1}) \m \Sing X$ and $X_{[s]} \m \Sing X$ are analytic submanifolds
of $X \m \Sing X$.
\end{enumerate}
\fin
\end{corollary}

%We immediately deduce from this lemma that the projection $\pi$ is a submersion on the manifold %$X\m V \times \bP^1$. Let us show some other consequences.

\subsection{The canonical stratification of $X$.}\label{ss:str}
 We use here the theory of Whitney stratified spaces, especially the existence of Whitney stratifications on any semi-analytic space. We do not require that a stratum is connected.
 Let us endow the analytic space $X$ with a Whitney regular stratification $\cS$, as follows.
Let us first recall that $X \subset B_{\e_0} \times \bP^{p-1}$ and all the analytic subsets of $X$ defined before are also well defined as
set germs at $\{ 0\} \times \bP^{p-1}$. In particular each fibre $X_{[s]}$ of $\pi$ can be viewed as a germ at the point $\{ 0\} \times [s]$.

We assume that either $0\in \bR^m$ is an isolated critical point or, more generally,  that $0\in \bR^p$ is an isolated critical value of $\psi$. Then Corollary \ref{c:submanif} shows that, in these both settings,
 $X \m (V \times \bP^{p-1})$ is an analytic manifold and an open dense subset of $X$---
this is defined as the highest dimensional stratum of $\cS$.
Next, we set $(V\times \bP^{p-1}) \m \Sing X$ as another stratum:  since it is an analytic submanifold of $X \m \Sing X$,  the couple of strata
$X \m (V \times \bP^{p-1})$ and $(V\times \bP^{p-1}) \m \Sing X$ automatically satisfy the Whitney \textbf{a}- and \textbf{b}-regularity conditions.

One may further stratify the analytic set $\Sing X$ such that the strata satisfy the Whitney regularity conditions.  By work of \L ojasiewicz and Mather  \cite{Lo, Ma1, Ma2}, there is a well-known algorithm to stratify
a semi-analytic set; this is described for instance \cite{GWPL}, see also \cite{Wa} for the main ideas. We get in this way the less fine (or, in other words, the most economical) Whitney stratification having $X \m (V \times \bP^{p-1})$ as the highest dimensional stratum.

\begin{definition} \label{d:str}
For a function germ $\psi$ with isolated critical value, let $\cS$ denote the Whitney regular stratification of $X$ defined above. We call it \textit{the canonical stratification of $X$}.
\end{definition}

Before giving the proof of the main theorem, let us look closer to the case $p=2$, which has been also addressed before by Jacquemard, Ruas, dos Santos and Seade in particular, see our references.

%%%%%%%%%%%%%%%%%%%%%%%%%%%%%%%%%%%%%%
%%%%%%%%%%%%%%
\subsection{The particular case $p=2$.}\label{ss:case_p=2}

 Denoting $\psi(x) := (P,Q)(x)$, our set $X$ is then\footnote{The fibres $X_{[s]}$ of $\pi$ have been considered from another viewpoint in \cite{RS, Se1, Se2} and denoted by $M_\theta$ or $X_\theta$, for some angle $\theta \in [0,\pi)$.} the following hypersurface in $B^m_{\e_0}\times \bP^1_{\bR}$:
\[ X := \{ (x, [s:t]) \in B^m_{\e_0}\times \bP^1_{\bR} \mid t\hspace{1pt}P(x) - s\hspace{1pt}Q(x) =0 \}.
\]
We then get the following precise description of the singular loci:
\begin{lemma}\label{l:sing2}
 $\Sing X = (V \times \bP^1) \cap \{ t \hspace{1pt}\nabla P(x) - s\hspace{1pt}\nabla Q(x) = 0\}$ and,  for any fixed $[s:t]\in \bP^1$,
$\Sing X_{[s:t]} = \{ t \hspace{1pt}\nabla P(x) - s\hspace{1pt}\nabla Q(x) = 0\}$.
\fin
\end{lemma}

\begin{remark}\label{r:cond}
 The assumption that the map (\ref{eq:milnormap2}) is a submersion for any $\e >0$ small enough,
 is equivalent to the fact that the vector $\omega(x) := P(x)\hspace{1pt} \nabla Q(x) - Q(x)\hspace{1pt} \nabla P(x)$ and the position vector $x$ are independent, for any $x\in B_\e\m V$ and small enough $\e >0$.

In other words, there exist some small $\e >0$ such that for any $x\in B_\e\m V$, we have the following inequality:
\begin{equation}\label{eq:main}
 \left| \left\langle \frac{\omega(x)}{\| \omega(x)\| }, \frac{x}{\|x\| } \right\rangle  \right| < 1.
\end{equation}

Furthermore, this is equivalent to the fact that $\bar M \not\supset \{ 0\}$, where
\begin{equation}\label{eq:main2}
 M:= \{ x\in \bR^m \m V \mid  \exists \lambda \not= 0 \mbox{ such that } P(x)\hspace{1pt} \nabla Q(x) - Q(x)\hspace{1pt}\nabla P(x) = \lambda x \}.
\end{equation}
\end{remark}
%%%%%%%%%%%%%%%%%%%%%%%%%%%%%%%%%%%%%%%%%%%%%%%%
%%%%%%%%%%%%%%%%%%%%%%%%%%%%%%%%%%%%%%%%

\section{Proof of Theorem \ref{t:main1}}

The case $V= \{ 0\}$ has been already discussed in \S \ref{intro}. It was also remarked that a necessary condition for (\ref{eq:milnormap2}) to define an open book decomposition is that this map is a submersion.

 Let's then suppose that $\dim V> 0$.
We have seen that, in case $\psi$ has an isolated critical point at $0\in \bR^m$, the canonical Whitney stratification $\cS$ defined at \ref{ss:str}
 consists of the following strata: $X \m (V \times \bP^{p-1})$,  $(V \m \{ 0\}) \times \bP^{p-1}$,  $\{ 0\} \times (\bP^{p-1} \m A)$ and the discrete points $(0,[s])$ where $[s]$ runs in some finite set $A\subset \bP^{p-1}$. More precisely, the set $A$ is a collection of finitely many points where the Whitney regularity conditions along the line $\{ 0\} \times \bP^{p-1}$ may eventually fail.

It follows from Whitney's finiteness theorem for germs of analytic sets, explained by Milnor in \cite[Appendix A]{Mi}, that there exists $\e_0 >0$ such that the sphere $S^{m-1}_\e$ is transversal to
$V$, for any $\e \le  \e_0$.
Take now such a small $\e_0$,
which in addition satisfies the criterion of Theorem \ref{t:main1}.
 This assumption precisely means that the sphere $S^{m-1}_\e \times \{[s]\}$ is
transversal to $X_{[s]}\m V$, for any $[s] \in \bP^{p-1}$ and for all $\e \le  \e_0$.
These facts, together with the observation that the subset  $\{ 0\} \times \bP^{p-1}$ does not intersect $S^{m-1}_\e \times \bP^{p-1}$, for $0< \e \le \e_0$, provide the proof of the following:

\begin{lemma}\label{l:basic}
 For any $\e\le \e_0$,  the ``torus'' $S^{m-1}_\e \times \bP^{p-1}$ is transversal to the strata of $\cS$. \fin
\end{lemma}

%In the remainder, we shall fix an $\e$ like in the above lemma.

A result from Whitney stratification theory \cite{Ma1, Ma2} tells that the transversal intersection of a manifold $M$
with a Whitney stratified space $Y$ induces a Whitney regular stratification on the intersection
$Y\cap M$. Applying this to $M= S^{m-1}_\e \times \bP^{p-1}$ we get that
our intersection $\cS \cap (S^{m-1}_\e \times \bP^{p-1})$ is a Whitney regular stratification
on $X \cap (S^{m-1}_\e \times \bP^{p-1})$ and we denote it by $\cS'$. Obviously, $\cS'$ has only two strata.

\begin{proposition}\label{p:subm}
If the criterion of Theorem \ref{t:main1} is satisfied, then for any small enough $\e$ the projection  $\pi_| : X\cap (S^{m-1}_\e \times \bP^{p-1}) \to \bP^{p-1}$ is a surjective stratified submersion relative to the Whitney stratification $\cS'$.
\end{proposition}
\begin{proof}
The restriction $\pi_|$ is a submersion on the open stratum. Indeed, this is equivalent to the transversality of the sphere $S^{m-1}_\e$ to $\pi^{-1}([s]) \m V = X_{[s]}\m V$, for any $[s] \in \bP^{p-1}$, which is the theorem's hypothesis.
On the second stratum, this is obvious since this stratum is a product manifold $K \times\bP^{p-1}$ by its definition.
Now, since $\pi_|$ is a proper map (hence a closed map) and a stratified submersion (hence an open map) it follows that it is surjective. By the Thom Isotopy Theorem one gets that
 $\pi_| : X\cap (S^{m-1}_\e \times \bP^{p-1}) \to \bP^{p-1}$ is a stratified fibration.
The proof of the Thom Isotopy Theorem yields a continuous vector field which is tangent to the strata, it is $C^\ity$ on each stratum, and it is a lift
of the unit tangent vector field on $\bP^{p-1}$ (with some chosen fixed orientation). This vector field trivialises the fibration.
\end{proof}

 By Proposition \ref{p:subm}, the map $\pi_|$ is a locally trivial fibration on each stratum of $X\cap (S^{m-1}_\e \times \bP^{p-1})$. In particular, the restriction  $\pi_| : (X\m (V \times \bP^{p-1}) ) \cap (S^{m-1}_\e \times \bP^{p-1}) \to \bP^{p-1}$ to the open stratum
is a locally trivial fibration and one has a trivialising vector field $\bv$ at each point, which is a lift of the continuous unit tangent vector field on $\bP^{p-1}$.

Let $j: S^{p-1} \to \bP^{p-1}$ be the standard double covering projection.
The unit vector field on $\bP^{p-1}$ lifts to the oriented unit vector field on $S^{p-1}$.
 Let now consider the isomorphism $i: S_\e^{m-1} \m K \to X\cap [(S_\e^{m-1} \m K)\times \bP^{p-1}]$,
where $i(x) := (x, [\psi(x)])$,
and the commuting diagram:
\[
 \begin{array}{ccc}
S_\e^{m-1} \m K & \stackrel{\frac{\psi}{\| \psi \|}}{\longrightarrow} & S^{p-1} \\
\downarrow i & \ & \downarrow j  \\
X\cap [(S_\e^{m-1} \m K)\times \bP^{p-1}] & \stackrel{\pi}{\longrightarrow} & \bP^{p-1}
\end{array}
\]
 The vector field $\bv$ lifts by $i$ to a vector field $\bw$
which is exactly the lift by $\frac{\psi}{\| \psi \|}$ of the oriented unit tangent vector field
on $S^{p-1}$. Therefore $\bw$ provides our trivialization.

This ends the proof of Theorem \ref{t:main1}. \fin
%%%%%%%%
%%%%%%%%%%%%%%%

\begin{example}\label{e:2}

$\psi : (\mathbb R^3 ,0) \to (\mathbb R^2 ,0)$, $\psi = (P,Q)$, where:
\[ P= z(x^2+y^2+z^2), \ \ \ \ Q=y-x^3 .\]

The function $\psi$ has isolated singularity at origin and the link is $K_{\e}=\{z=0, y=x^3\}\cap S^{2}_{\e}$, i.e. two points.

We have $\nabla P(x,y,z)=(2xz,2yz,x^2+y^2+3z^2)$, $\nabla Q(x,y,z)=(-3x^2,1,0)$. Finding the set $M$ from (\ref{eq:main2}) amounts to computing the $2\times 2$ minors of the following $2\times 3$ matrix:
\[\left(
\begin{array}{ccc}
-x^2z(x^2+3y^2+3z^2)-2yxz & z(x^2+z^2-y^2)+2yzx^3 & -(y-x^3)(x^2+y^2+3z^2) \\
 x & y & z \\
\end{array}
\right)
\]
Pursuing the computations for some time we finally find that $\bar M \not\supset \{ 0\}$, which shows that
(\ref{eq:milnormap2}) yields an open book decomposition, by Theorem \ref{t:main1}.
\end{example}

\begin{remark}\label{r:thmain}
 In the case $p=2$ (see \S \ref{ss:case_p=2} for the notations), Jacquemard \cite{Ja1, Ja2} exhibited a sufficient condition for the existence of a ``strong Milnor fibration'' (\ref{eq:milnormap2}): in some small neighbourhood of the origin the vectors $\nabla P$ and $\nabla Q$ should not be parallel and they  should have equal order on any analytic path through the origin. The holomorphic functions clearly satisfy these conditions but the first examples of genuine real germs were given later in \cite{RSV}.
Next \cite{RS} provides another sufficient condition of different nature, based on Bekka's (c)-regularity. The authors show that their condition is weaker than Jacquemard's since the latter implies Verdier's (w)-regularity. Moreover, \cite{dS} gives another
sufficient condition, showing that it is weaker than the one in \cite{RS} and providing an example satisfying it but not the (c)-regularity requirement from \cite{RS}.
 Our Theorem \ref{t:main1} recovers all these results since it provides a necessary and sufficient condition. One may also check that our condition is weaker than the condition
given in \cite{dS}.
\end{remark}

\section{Weighted-homogeneous map germs}\label{weighted-hom}
%%%%%%%%%%%%%%%%%%%%%%%%%%%%%%%%%%%%%%%%
%%%%%%%%%%%%%%%%%%%%%%%%%%%%%%%%%%%%%%%%%
 We have seen up to now that Theorem \ref{t:main1} completely answers the question (A)
in case $\Sing \psi \subset \{ 0\}$. We focus below on the class of weighted-homogeneous map germs
and show that this answers affirmatively to both questions (A) and (B) in \S \ref{intro}.

We keep the notations introduced in \S \ref{intro} and \S \ref{prep}. One says that $\psi = (P_1, \ldots , P_p) : (\mathbb R^m ,0) \to (\mathbb R^p ,0)$ is a \textit{weighted-homogeneous map germ} of weights $(w_1, \ldots , w_m)\in \bN^m$ and of degree $\alpha \in \bN$ if all the functions $P_i$ are weighted-homogeneous of degree $\alpha$. Let us denote by $\gamma (x) := \sum_{j=1}^m w_j x_j \frac{\partial}{\partial x_j}= (w_1 x_1, \ldots , w_m x_m)$ the Euler vector field on $\bR^m$. By definition we have $\langle \nabla P_i(x), \gamma(x) \rangle = \alpha P_i(x)$ for all $i$.

With these notations, we prove the following result:
\begin{theorem}\label{t:quasi}
Let $\psi :(\mathbb{R}^{m},0) \to (\mathbb{R}^{p},0)$ be a weighted-homogeneous map germ with  isolated singularity. Then the two open book decompositions defined by (\ref{eq:milnormap1}) and
(\ref{eq:milnormap2}) exist and are equivalent.
\end{theorem}

\begin{proof}
\textit{Step 1.} As discussed in \S \ref{intro}, the existence of the first type open book decomposition is insured by Milnor's theorem \cite[p. 97]{Mi} since the map $\psi$ has isolated singularity.

 We next use Theorem \ref{t:main1} to prove the existence of the fiber bundle defined by the map (\ref{eq:milnormap2}). In order to show that our $\psi$ satisfies the criterion of Theorem \ref{t:main1} it is enough to prove that $X_{[s]}$ is transversal to the sphere $S^{m-1}_\e$ for any $[s]\in \bP^{p-1}$ and any small enough $\e >0$.
In other words, we need to show that $X_{[s]}$ is not tangent to the sphere. To do that, we shall first prove that the Euler vector field $\gamma(x)$ is tangent to $X_{[s]}$ for any fixed $[s]$ and then that the normal vector to the sphere is not orthogonal to $\gamma(x)$. This will clear our claim.

As shown in \S \ref{ss:sing}, the analytic set $X_{[s]}$ is defined by $p-1$
independent equations $f_{ij} := s_j P_i - P_j =0$, for $j\not= i$, and it is non-singular
at all points except of the origin.
Let $(x, [s])\in X_{[s]}$, with $x\not= 0$, in some chart $\{ s_i=1\}$ of $\bP^{p-1}$. Let $\nabla_x$ denote the gradient
with respect to the coordinates of $\bR^m$. Then we have:
\[ \langle \gamma(x),  \nabla_x f_{ij}(x) \rangle = \alpha f_{ij}(x) = 0
 \]
for all $j\not= i$, which means that $\gamma(x)$ is tangent to $X_{[s]}$.

We also have:
 \[ \langle \gamma(x),  x \rangle = \sum_i w_i x_i^2 > 0,
 \]
since $x \not= 0$. This shows that the vector $x$ cannot be orthogonal to the tangent space of $X_{[s]}$, hence the sphere $S^{m-1}_\e$ is indeed transversal to $X_{[s]}$, for any $[s]\in \bP^{p-1}$ and actually for any $\e >0$. Our claim is proved. Besides, the proof shows that
the size of the spheres  $S^{m-1}_\e$ is not important, due to the weighted-homogeneity.

\smallskip
\noindent
\textit{Step 2.}
 In the first part of his general proof of the existence of the first type open book decomposition, Milnor \cite[p. 97-98]{Mi} has shown the following fact:  since $\Sing \psi \subset \{ 0\}$, for all small enough $\eta >0$,  the restriction:
\begin{equation}\label{eq:tube}
 \psi_| : \bar B^m_{\e} \cap \psi^{-1}(\bar B^{p}_\eta \m \{ 0\}) \to \bar B^{p}_\eta \m \{ 0\}
\end{equation}
 is a fiber bundle, the fibre of which is a manifold with boundary. This fibration does not depend on the choices of $0< \e \le \e_0$ and of $0< \eta \ll \e$.

Then Milnor shows that the ``tube''
$\bar B^m_{\e} \cap \psi^{-1}(S^{p-1}_\eta)$ can be ``inflated'' to the sphere $S^{m-1}_{\e}$
by using a vector field without zeroes and pointing ``outwards''. This procedure induces a fibration
on $S^{m-1}_{\e} \m \psi^{-1}(B^{p}_\eta)$ which is isomorphic (up to diffeomorphism) to the fibration:
\begin{equation}\label{eq:tube2}
 \psi_| : \bar B^m_{\e} \cap \psi^{-1}(S^{p-1}_\eta) \to S^{p-1}_\eta
\end{equation}
 for any choice of the vector field with the good properties. These fibrations coincide
 on $S^{m-1}_{\e} \cap \psi^{-1}(S^{p-1}_\eta)$.

Milnor's proof quoted above is valid in whole generality, for maps $\psi$ with isolated singularity. In our quasi-homogeneous setting, one has the Euler vector field and it appears that this is particularly convenient.
Indeed, we have the following inequalities, for any $x\not\in V$:
\[\langle \gamma (x), x\rangle >0 , \]
\[\langle \gamma (x), \nabla (\| \psi(x) \|)^{2}\rangle >0 ,\]
where the former was proved just above and the latter is again due to the weighted-homogeneity.
 These inequalities show that the Euler vector field  $\gamma (x)$ is without zeroes and points outwards as required.
Moreover, the Euler vector field is tangent to the
fibres of the map $\frac{\psi}{\| \psi\|}$.

This shows that the fibration induced
by the ``inflating'' procedure is precisely the fibration (\ref{eq:milnormap2}).
 Therefore we have proved that the two types
of open book decompositions, (\ref{eq:milnormap1}) and (\ref{eq:milnormap2}), coincide
on $S^{m-1}_{\e} \m \psi^{-1}(B^{p}_\eta)$. Now the fibration (\ref{eq:tube}) shows that
on $S^{m-1}_{\e} \cap \psi^{-1}(\bar B^{p}_\eta)$ the map $\frac{\psi}{\| \psi\|}$
is a locally trivial fibration. This observation completes the open book decompositions.
\end{proof}

%%%%%%%%%%%%%%%%%%%%%%%%%%%%%%%%%%%%%%%%%%%%%%%%%%%%

%%%%%%%%%%%%%%%%%%%%%%%%%%%%%%%%%%%%%%%%%%%%%%%%%%%
%%%%%%%%%%%%%%%%%%%%%%%%%%%%%%%%%%%%%%%%%%%%%%%%%%%%

\begin{example}\label{e:3}
A'Campo's \cite{ac} example $\psi : (\mathbb{R}^4, 0) \to (\mathbb{R}^{2}, 0)$, $\psi(x,y,z,w)=(x(z^2+w^2)+z(x^2+y^2),y(z^2+w^2)+w(x^2+y^2))$ is weighted-homogeneous. By our Theorem \ref{t:quasi}, the fibration (\ref{eq:milnormap2}) on the 3-sphere given by $\frac{\psi}{\|\psi \|}$ is equivalent to Milnor's fibration (\ref{eq:milnormap1}).
\end{example}
%%%%%%%%%%%%%
%%%%%%%%%%%%%%%%%%%%%%%%%%%%%%%%%%%%%%%%
%%%%%%%%%%%%%%%%%%

\section{Isolated critical value}\label{critvalue}

We raise generality and address the questions (A) and (B) in the setting of maps
$\psi : (\bR^m, 0) \to (\bR^p,0)$ with an isolated critical value at $0\in \mathbb{R}^{p}$. We derive below some sufficient conditions for the existence of the open book decompositions (\ref{eq:milnormap1}) and (\ref{eq:milnormap2}).

As we have recalled in the proof of Theorem \ref{t:quasi}, Milnor's proof of the existence of the locally trivial fibration (\ref{eq:milnormap1}) contains two steps.
The first one needs that the map (\ref{eq:tube}) is a locally trivial fibration, for all $0< \e \le \e_0$ and $0< \eta \ll \e$.
The second step consists of ``inflating'' the tube to the sphere, as explained in the first part of Step 2 of the proof of Theorem \ref{t:quasi}. This can be done in exactly the same way whenever $\psi$ has an isolated critical value instead an isolated critical point. We then have proved the following result.

\begin{proposition}\label{p:isolcrtval}
Let $\psi$ have an isolated critical value at $0$. Suppose that (\ref{eq:tube})
is a locally trivial fibration. Then there exists an open book decomposition (\ref{eq:milnormap1}).
\fin
\end{proposition}

\begin{remark}\label{r:ps}
It is well-known that if $V = \psi^{-1}(0)$ has a Whitney stratification $\displaystyle{\xi}$ such that in a sufficient small ball $B_{\epsilon}$, the pair $(B_{\epsilon}\setminus V, S)$ satisfies Thom's condition (a$_\psi$), for all $S\in \xi $, then (\ref{eq:tube}) is a locally trivial fibration. This fact attached to Proposition \ref{p:isolcrtval} yield a statement proved by Pichon and Seade \cite[Theorem 1.1]{PS}.
\end{remark}

In the complex setting $f: (\bC^n, 0) \to (\bC,0)$, L\^e D.T. \cite{Le} (see also \cite{HL}) showed that there exists a fibration in a thin tube around $V$ due to the Thom (a$_f$) property which is satisfied by
the holomorphic function germs. Moreover, the fibres of the map $\frac{f}{\| f\|}$ are transversal to small enough spheres and provide an open book decomposition, by
Milnor's proof in \cite{Mi}.

 For real maps, we have the following byproduct of the proof of Theorem \ref{t:main1}:

\begin{proposition}\label{p:isolcrt}
 Let $\psi$ have an isolated critical value at $0$. Suppose that, for any small enough $0< \eta \ll \e$, the map:
\begin{equation}\label{eq:tube3}
 \psi_| : S^{m-1}_{\e} \cap \psi^{-1}(\bar B^{p}_\eta \m \{ 0\}) \to \bar B^{p}_\eta \m \{ 0\}
\end{equation}
is a locally trivial fibration in a thin tube and
suppose that the criterion of Theorem \ref{t:main1} is satisfied. Then $\psi$ has a ``strong Milnor fibration'', i.e. (\ref{eq:milnormap2}) is an open book decomposition.
\end{proposition}
\begin{proof}
The hypothesis of Theorem \ref{t:main1} insures that the fibres of the map
$\frac{\psi}{\| \psi\|}$ are transversal to the small spheres. In addition, we need to control the map $\frac{\psi}{\| \psi\|}$ at the points of the link $K$ (if this is non-empty) since $\frac{\psi}{\| \psi\|}$ is non-proper and one cannot apply Ehresmann's theorem. The device we may use in order to control it is the fibration structure of the map $\psi$ near $V$.

For having a controlled local trivialization of $\frac{\psi}{\| \psi\|}$
in the neighborhood of $K$, it is sufficient to see how this can be constructed locally. Take some tangent vector field $\bv$ defined on some small open subset $U$ of $S^{m-1}_{\e}$.
 Then extend it to the cone over $U$ minus its vertex at the origin
to the vector field $\bw(x) := (\rho_t)_*(\bv(x))$, where $\rho : S^{m-1}_{\e} \times (0, 1] \to \bar B^{p}_\eta \m \{ 0\}$, $\rho(x,t) := tx$ and $\rho_t(x):= \rho(x,t)$.

Then using a local trivialization of the fibration (\ref{eq:tube3}),
lift $\bw$ to a vector field in the tube such that it is tangent to the boundary tubes $S^{m-1}_\e \cap \psi^{-1}(S^{p-1}_\nu)$, for any $0< \nu \le \eta$.
Then the flow of $\bw$ provides a local trivialization of $\frac{\psi}{\| \psi\|}$.
\end{proof}

\begin{remark}
Again the existence of a Whitney stratification on $K\subset S^{m-1}_\e$
having Thom's (a$_\psi$) property for the restriction of $\psi$ to $S^{m-1}_\e \cap \psi^{-1}(\bar B^{p}_\eta)$ insures that (\ref{eq:tube3}) is a locally trivial fibration.
\end{remark}

%\subsection{More developments: 1. meromorphic germs and
%2. fibrations at infinity of polynomial maps.}

%%%%%%%%%%%%%%
%%%%%%%%%%%%%%%%%%%%%%

%%%%%%%%%%%%%%%%%%%%%%
%%%%%%%%%%%%%%%%%%%%%%
%%%%%%%%%%%%%%%

\end{document}